\documentclass[11pt,reqno]{amsart}
\usepackage{amsthm, amsmath, amssymb, amscd, latexsym, multicol, verbatim, enumerate, }
\usepackage{hyperref}
\advance\textwidth by 1.2in \advance\oddsidemargin by -.6in \advance\evensidemargin by -.6in \parskip=.1cm
\theoremstyle{definition} 
\newtheorem{coro}{Corollary}
\newtheorem{lem}{Lemma}
\newtheorem{prop}{Proposition}
\newtheorem{thm}{Theorem}

\newtheorem*{thm*}{Theorem}
\newtheorem{theom}{Theorem}
\newtheorem*{lem*}{Lemma}
\newtheorem*{prop*}{Proposition}
\newtheorem*{coro*}{Corollary}

\newcounter{cnt}
 \makeatletter
\def\mydggeometry{\makeatletter\dg@YGRID=1\dg@XGRID=20\unitlength=0.003pt\makeatother}
\makeatother \theoremstyle{remark}
\numberwithin{equation}{section}
\newcommand{\nc}{\newcommand}
\nc{\lie}[1]{\mathfrak{#1}} \nc\bp{\mathbf{p}} \nc\bq{\mathbf{q}} \nc\mD{\mathbb{D}} \nc\bs{\mathbf{s}} \nc\bt{\mathbf{t}} \nc\bz{\mathbb{Z}} \nc\bc{\mathbb C} 
\nc\Max{\operatorname{Max}}

\begin{document}
\author[Fourier]{Ghislain Fourier}
\address{Mathematisches Institut, Universit\"at Bonn}
\address{School of Mathematics and Statistics, University of Glasgow}
\email{ghislain.fourier@glasgow.ac.uk}

\thanks{G.F. has been supported by the DFG priority program ''Representation Theory'' 1388}

\subjclass[2000]{Primary: 52C99, 14M25, 06A07, 52B20, 14M15, 17B10}
\keywords{PBW filtration, Polytope, Marked Poset, Chain Polytope}
\date{\today}

\title[Marked posets]{Marked poset polytopes: Minkowski sums, indecomposables, and unimodular equivalence}
\begin{abstract}
We analyze marked poset polytopes and generalize a result due to Hibi and Li, answering whether the marked chain polytope  is unimodular equivalent to the marked order polytope. Both polytopes appear naturally in the representation theory of semi-simple Lie algebras, and hence we can give a necessary and sufficient condition on the marked poset such that the associated toric degenerations of the corresponding partial flag variety are isomorphic.\\
We further show that the set of lattice points in such a marked poset polytope is the Minkowski sum of sets of lattice points for 0-1 polytopes. Moreover, we provide a decomposition of the marked poset into indecomposable marked posets, which respects this Minkowski sum decomposition for the marked chain polytopes polytopes.

\end{abstract}
\maketitle

\section{Introduction and main results}
Let $(\mathcal{P}, \prec)$ be a finite poset and $A$ subset that contains at least all maximal and all minimal elements of $\mathcal{P}$. We set
\[
\mathcal{Q}_A := \{ \lambda \in \bz_{\geq 0}^{A} \mid \lambda_a \leq \lambda_b \text{ if } a \prec b \}
\]
and call the triple $(\mathcal{P}, A, \lambda)$ a \textit{marked poset}. Then the marked chain polytope associated to $\lambda \in \mathcal{Q}_A$ is defined in \cite{ABS11}:
\[
\mathcal{C}(\mathcal{P}, A)_\lambda := \{ \bs \in \mathbb{R}_{\geq 0}^{\mathcal{P}\setminus A} \mid \bs_{x_1} + \ldots +  \bs_{x_n} \leq \lambda_b - \lambda_a \text{ for all chains }  a \prec x_1 \prec \ldots \prec x_n \prec b\},
\]
while the marked order polytope is defined as
\[
\mathcal{O}(\mathcal{P}, A)_\lambda := \{ \bs \in \mathbb{R}_{\geq 0}^{\mathcal{P}\setminus A} \mid \bs_{x} \leq  \bs_{y}\, , \, \lambda_{a} \leq  \bs_{x} \leq \lambda_{b}\, , \, \text{ for all } a \prec x \prec b, x \prec y \},
\]
where $a,b\in A, x_i,x,y \in \mathcal{P}\setminus A$. This generalizes the notion of chain and order polytopes due to Stanley \cite{Sta86}, where $A$ consists exactly of all extremal elements and $\lambda_b = 1$ for all maximal elements, $\lambda_a = 0$ for all minimal elements.\\
If $\lambda = c \mathbf{e}_A$, e.g. $\lambda_a = \lambda_b = c$, then the marked chain polytope is just the origin, while the marked order polytope is just the point $\mathbf{e}_{\mathcal{P}\setminus A}$. 
Suppose $c = \operatorname{min } \{ \lambda_a | a \in A \}$, then the marked poset polytopes associated with $(\mathcal{P}, A, \lambda)$ are just affine translations of the marked poset polytopes associated with $(\mathcal{P}, A, \lambda - c \mathbf{e}_A)$.
So for the study of marked poset polytopes, it is enough to consider the following marking vectors: 
\[
(\mathcal{Q}_A )_0 = \{ \lambda \in \mathcal{Q}_A \mid \exists \; a \in A\, : \, \lambda_a = 0 \}.
\]

We call the marked poset $(\mathcal{P}, A, \lambda)$ regular if for all $a \neq b$  in $A :\lambda_a \neq \lambda_b$ and there are no \textit{obviously redundant} relations (see Section~\ref{Section-Facets} for details). Any given non-regular marked poset can be transformed into a regular marked poset by transforming certain cover relations, the transformations of the associated marked poset polytopes are nothing but affine translations.\\

The first relation between the two marked poset polytopes we should point out here, is due to Ardila-Bliem-Salazar, who showed that both polytopes have the same Ehrhart polynomial, hence do have the same number of lattice points \cite{ABS11}. They provide a piecewise linear bijection (unfortunately, in terms of the representation theoretical interpretation (Section~\ref{Section-RT}), this map does not preserve the weight of a lattice point).\\
Let $(\mathcal{P}, A, \lambda)$ be a regular marked poset, then the number of facets in the marked order polytope is equal to the number of cover relations in $\mathcal{P}$ (Lemma~\ref{lem-num-facets}). Further, if we denote for each pair $a \prec b$ in $A$ the number of saturated chains $a \prec x_1 \prec \ldots \prec x_p \prec b$, $x_i \in \mathcal{P}\setminus A$, by $c(a,b)$, then the number of facets in the marked chain polytope is equal 
\[
|\mathcal{P}\setminus A| + \sum_{a \prec b} c(a,b).
\]
We say that a poset has a \textit{star relation} if there is a subposet whose Hasse diagram is of the form (recall that we draw a downward arrow from $y$ to $x$ if $x \prec y$)\\
\begin{figure}[ht]
\begin{picture}(300,55)
\put(125,10){$\circ$}
\put(132,49){$\line(1,-1){10}$}
\put(132,17){$\line(1,1){10}$}
\put(125,50){$\circ$}
\put(145,30){$\circ$}
\put(163,17){$\line(-1,1){10}$}
\put(165,10){$\circ$}
\put(165,50){$\circ$}
\put(163,49){$\line(-1,-1){10}$}
\end{picture}
\caption{The star relation}
\end{figure}
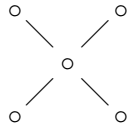
\begin{theom}
Let $(\mathcal{P}, A, \lambda)$ be a regular marked poset, then 
\[
\sharp \{ \text{ facets in } \mathcal{O}(\mathcal{P}, A)_\lambda \} \leq \sharp \{ \text{ facets in } \mathcal{C}(\mathcal{P}, A)_\lambda \}
\]
and equality if and only if $\mathcal{P}$ has no star relation.
\end{theom}\label{main-thm}
This, and the following theorem on equivalent statements on the marked poset polytopes generalizes results due to Hibi and Li (\cite{HL12}) for the special case $\lambda_b \in \{0,1\}$ for all $b \in A$:
\begin{theom}\label{thm-equiv} Let $(\mathcal{P}, A, \lambda)$ be a regular marked poset, then the following are equivalent
\begin{enumerate}
\item $\mathcal{O}(\mathcal{P}, A)_\lambda $ and $\mathcal{C}(\mathcal{P}, A)_\lambda$ are unimodular equivalent.
\item $\mathcal{O}(\mathcal{P}, A)_\lambda $ and $\mathcal{C}(\mathcal{P}, A)_\lambda$ have the same f-vector, e.g. the number of $i$-dimensional faces in both polytopes is the same.
\item $\mathcal{O}(\mathcal{P}, A)_\lambda $ and $\mathcal{C}(\mathcal{P}, A)_\lambda$ have the sum number of facets.
\item $\mathcal{P}$ has no star relation.
\end{enumerate}
\end{theom}
Certainly, we have to prove $(3) \Rightarrow (4) \Rightarrow (1)$ only, this will be done in Section~\ref{Section-Facets} and Section~\ref{Section-Uni}.

\noindent
In the second part of the paper we are mainly interested in the lattice points of the marked poset polytopes:
\[
S_{\mathcal{C}}(\lambda) := \mathcal{C}(\mathcal{P}, A)_\lambda \cap \bz^{\mathcal{P}\setminus A} \, , \,  S_{\mathcal{O}}(\lambda) := \mathcal{O}(\mathcal{P}, A)_\lambda \cap \bz^{\mathcal{P}\setminus A}.
\]
$\lambda \in (\mathcal{Q}_A)_0$ is called  $\mathcal{C}$-\textit{indecomposable} if for all $\tau, \mu \in \mathcal{Q}_A$:
\[
S_{\mathcal{C}}(\lambda) = S_{\mathcal{C}}(\mu) + S_{\mathcal{C}}(\tau) \Rightarrow \mu = 0 \vee \tau = 0,
\]
To characterize the indecomposable marked chain polytopes we introduce the notion of reduced poset. Let $\lambda \in \mathcal{Q}_A$, then we denote the full subposet
\[
\overline{(\mathcal{P}, A, \lambda)} = \{ x \in \mathcal{P} \mid \not\exists a \in A \text{ s.t. } \lambda_a = 0 \text{ and } x \preceq a \}.
\]
So we subtract from $\mathcal{P}$ all elements which are bounded above by $0$, especially we subtract all elements from $A$ which are marked by $0$. Note that $\overline{(\mathcal{P}, A, \lambda)}$ is not a marked poset anymore. 
See the following example, the marked poset and the reduced poset:
\begin{figure}[ht]
\begin{picture}(450,100)
\put(95,100){$1$} 
\put(100,98){$\line(1,-1){10}$}
\put(110,82){$x$}
\put(115,80){$\line(1,-1){10}$}

\put(125,59){$0$}
\put(127,57){$\line(0,-1){10}$}
\put(125,38){$w$}
\put(127,36){$\line(0,-1){10}$}
 \put(125,17){$0$}

\put(157,98){$\line(-1,-1){10}$}
\put(155,100){$1$}
\put(160,98){$\line(1,-1){10}$}
\put(142,80){$\line(-1,-1){10}$}
\put(141,82){$y$}
\put(187,98){$\line(-1,-1){10}$}
\put(170,82){$z$}
\put(172,80){$\line(0,-1){10}$}

\put(185,100){$1$}\put(170,59){$0$}
\put(215, 50){$\Longrightarrow$}\put(255,60){$1$}\put(255,42){$x$}\put(257,58){$\line(0,-1){10}$}\put(310,58){$\line(1,-1){10}$}
\put(305,60){$1$}
\put(307,58){$\line(-1,-1){10}$}
\put(290,42){$y$}
\put(320,42){$z$}
\put(337,58){$\line(-1,-1){10}$}
\put(335,60){$1$}
\end{picture}\\
\caption{The reduced poset}
\end{figure}
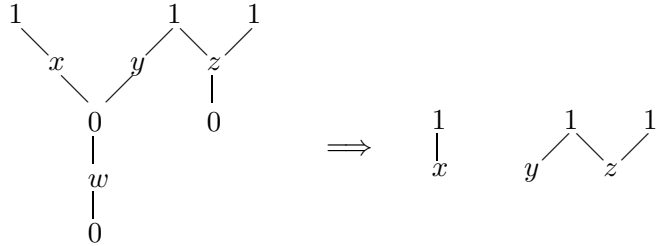

The other main result in this paper is the following:
\begin{theom}\label{main-thm2}
Let $\lambda \in \mathcal{Q}(A)_0$, then the following are equivalent:
\begin{enumerate}
\item $\lambda$ is $\mathcal{C}$-indecomposable.
\item $\lambda_a \in \{0,1\}$ for all $a\in A$ and the Hasse diagram of the reduced poset $\overline{(\mathcal{P}, A, \lambda)}$ is connected.
\end{enumerate}
\end{theom}
The proof will be given in Section~\ref{Section-Inde}. We further provide a decomposition $\mu_1 + \ldots + \mu_s$ of any given $\lambda \in \mathcal{Q}_A$ such that 
\[
S_{\mathcal{C}}(\lambda) = \sum S_{\mathcal{C}}(\mu_i) 
\]
and each summand corresponds to the lattice points of an indecomposable marked chain polytope. We see that is sufficient to compute the lattice points for the indecomposable marked chain polytopes and ''control'' the Minkowski sum of these lattice points.\\ 
In \cite{ABS11}, a piecewise linear bijection between marked order and marked chain polytopes has been provided, using this bijection one obtains now the lattice points of the marked order polytope.\\
Our main motivation for this last result comes from the representation theory of semi-simple Lie algebras, as we will explain in Section~\ref{Section-RT}. The idea of this Minkowski sum decomposition is basically the reduction of the computation of a monomial basis (parametrized by lattice points in a marked chain polytope) of an arbitrary simple representation to minimal representations whose monomial bases should be parametrized by lattice points in an indecomposable marked chain polytope.

\textbf{Acknowledgments} The author is funded by the DFG priority program ''Representation theory'' and would like to thank the University Bonn as well as the University of Cologne for their hospitality.

\section{Motivation from representation theory}\label{Section-RT}
We recall here several interesting examples of marked poset polytopes, arising in representation theory. For this let us recall briefly the PBW filtration and the associated degenerations. Let $\lie u$ be a Lie algebra with universal enveloping algebra  $U(\lie u)$. The PBW filtration on $U(\lie u)$ is given by
\[
U(\lie u)_s := \langle x_{i_1} \cdots x_{i_\ell} \mid x_{i_j} \in \lie u\, , \, \ell \leq s \rangle.
\] 
The associated graded algebra is isomorphic to the symmetric algebra of the vector space $\lie u$, $\operatorname{gr} U(\lie u) = S(\lie u)$. Let $M$ be a $\lie u$-module and $G \subset M$ a generating set, then the PBW filtration induces a filtration of $M$
\[
M_s := U(\lie u)_s. G
\]
whose associated graded module is a module for $S(\lie u)$ (and not for $\lie u$, unless $\lie u$ acts commutative of $M$). \\
This construction has been studied for affine Lie algebras in\cite{Fei09, CF13, FM14}). 
We will focus here on finite dimensional, simple Lie algebras with triangular decomposition $\lie g = \lie n^+ \oplus \lie h \oplus \lie n^-$ and finite-dimensional, simple modules $V(\lambda)$, where $\lambda$ is the highest weight. Then $V(\lambda) = U(\lie n^-).v_\lambda$, where $v_\lambda$ is a highest weight vector.\\ 
Further, if $v_{w \lambda}$ is an extremal weight vector, then the Demazure module $V_w(\lambda)$ is defined as the $\lie n^+ \oplus \lie h$-submodule generated by $v_{w\lambda}$. In any case, the associated graded module will be denoted $\operatorname{gr} M$ and it is either a $S(\lie n^-)$-module or a $S(\lie n^+ \oplus \lie h)$-module.\\
In the framework of PBW graded modules for simple Lie algebras (initiated in \cite{FFoL11a}, for quantum groups in \cite{FFR15}), one important issue is the construction of a monomial basis of the associated graded module with a ''nice'' combinatorial description in terms of lattice points in a normal polytope (see \cite{FFoL11a, FFoL11b, G11, BD14, BF14, Fou14b, Fou14}). \\
On the other hand, a connection of the graded character to symmetric MacDonald polynomials has been conjectured in \cite{CF13} and partially established in \cite{FM14} (see also \cite{BBDF14} for the degree of the Hilbert-Poincar\'e polynomial).\\
In \cite{Fei12}, this filtration has been used to construct degenerations of flag varieties and corresponding desingularizations \cite{FF13, FFL11}. 
Following this approach, in \cite{FFoL13a}, the term of a favourable module has been introduced, important geometric properties of the module are governed by a normal polytope.  We recall here in which cases marked chain polytopes arise in the context of PBW degenerations.

\subsection{The \texorpdfstring{$\lie{sl}_{n+1}$}{special linear Lie algebra}-case}
For each $n$ and $\lambda = (\lambda_1 \geq \ldots \geq \lambda_n \geq 0)$, let $(\mathcal{P}, A, \lambda)$ be the marked poset with Hasse diagram Figure~\ref{fig-sln}.\\
\begin{figure}[ht]
\begin{picture}(200,180)
\put(5,10){$0$}\put(20,10)
{$\bullet$}\put(20,40){$\bullet$}
\put(5,40){$\lambda_n$}
\put(20,70){$\vdots$}
\put(20,100){$\bullet$}
\put(5,100){$\lambda_3$}
\put(20,130){$\bullet$}
\put(5,130){$\lambda_2$}
\put(20,160){$\bullet$}
\put(5,160){$\lambda_1$}
\put(40,25){$\circ$}\put(40,55){$\circ$}\put(40,85){$\circ$}\put(40,115){$\circ$}\put(40,145){$\circ$}
\put(60,40){$\circ$}\put(60,70){$\circ$}\put(60,100){$\circ$}\put(60,130){$\circ$}
\put(80,55){$\circ$}\put(80,85){$\circ$}\put(80,115){$\circ$}
\put(100,70){$\circ$}\put(100,100){$\circ$}
\put(120,85){$\circ$}
\put(25,160){$\line(4,-3){15}$}\put(25,130){$\line(4,-3){15}$}\put(25,100){$\line(4,-3){15}$}\put(25,40){$\line(4,-3){15}$}
\put(45,145){$\line(4,-3){15}$}\put(45,115){$\line(4,-3){15}$}\put(45,85){$\line(4,-3){15}$}\put(45,55){$\line(4,-3){15}$}
\put(65,130){$\line(4,-3){15}$}\put(65,100){$\line(4,-3){15}$}\put(65,70){$\line(4,-3){15}$}
\put(85,85){$\line(4,-3){15}$}\put(85,115){$\line(4,-3){15}$}
\put(105,100){$\line(4,-3){15}$}
\put(60,130 ){$\line(-4,-3){15}$}\put(60,100 ){$\line(-4,-3){15}$}\put(60,70){$\line(-4,-3){15}$}\put(60,40 ){$\line(-4,-3){15}$}
\put(40,145 ){$\line(-4,-3){15}$}\put(40,115 ){$\line(-4,-3){15}$}\put(40,55 ){$\line(-4,-3){15}$}\put(40,25 ){$\line(-4,-3){15}$}
\put(80,115){$\line(-4,-3){15}$}\put(80,85 ){$\line(-4,-3){15}$}\put(80,55 ){$\line(-4,-3){15}$}
\put(100,100 ){$\line(-4,-3){15}$}\put(100,70 ){$\line(-4,-3){15}$}
\put(120,85){$\line(-4,-3){15}$}
\end{picture}
\caption{The $\lie{sl}_{n+1}$-poset}\label{fig-sln}
\end{figure}
To be precise, let $\mathcal{P} = \{ x_{i,j} \mid 0 \leq i \leq n, 1 \leq j \leq n+1 - i \}$ and cover relations $x_{i-1, j+1} \geq x_{i,j} \geq x_{i-1, j}$. Let $\lambda = (\lambda_1 \geq \ldots \geq \lambda_n \geq \lambda_{n+1} = 0)$ and let $x_{0,j}$ be marked with $\lambda_j$. \\
We have (Lemma~\ref{lem-num-facets}):
\[
|\{ \textit{facets in the marked order polytope} \} |= n(n+1)
\]
\[
|\{ \textit{facets in the marked chain polytope} \} | = n(n-1)/2 + \sum_{i=1}^{n} {i} C_{n-i}
\]
where $C_{n-i}$ is the Catalan number.\\

The associated marked order polytope is known as the Gelfand-Tsetlin polytope associated with the partition $\lambda$ (\cite{GT50}). 
The set of lattice points in this polytope parametrizes a basis of the simple $\lie{sl}_{n+1}$-module with highest weight $\lambda$ (the basis is obtained by restricting $V(\lambda)$ to a specific subalgebra isomorphic to $\lie{sl}_n$, and iterating this to $\lie{sl}_2$).\\

The marked chain polytope parametrizes a basis of the associated PBW-graded module $\operatorname{gr}V(\lambda)$ (\cite{FFoL11a}).\\

Both polytopes are normal and induce flat degenerations of (partial) flag varieties $SL_{n+1}/P_{\lambda}$ for the parabolic subalgebra associated to $\lambda$  (see \cite{FFoL13a} for details on the degenerated flag varieties associated to the marked chain polytope and \cite{AB04} for the degenerated flag varieties associated to the marked oder polytopes). 
It is very natural to ask wether the two degenerated toric flag varieties are isomorphic, and it has been shown, by providing some examples using \textit{polymake} \cite{GJ00}, that this is not true in general. Finally, by using Theorem~\ref{thm-equiv}, we can be very precise:
\begin{prop}
The two toric degenerations of $SL_n/P_\lambda$ are isomorphic if and only if one of the following is satisfied (set $m_i := \lambda_i - \lambda_{i+1}$):
\begin{enumerate}
\item $m_i = 0$ for $i=3, \ldots, n$.
\item $m_i = 0$ for $i=1, \ldots, n-2$.
\item $m_i = 0$ for $i=2, \ldots, n-1$.
\end{enumerate}
\end{prop}
\begin{proof}
Suppose $m_i \neq 0$, then we can read from the poset Figure~\ref{fig-sln}, that there is a subposet which has the form of a rectangle with $i$-rows and $n+1-i$-columns. This implies that for $3 \leq i \leq n-2$, there is a star subposet if $m_i \neq 0$. Now, suppose $n=3$ and $m_1, m_2, m_3 \neq 0$, then the poset has the form as in Figure~\ref{fig-sl4}, and so it has a star subposet.

\begin{figure}[ht]
\begin{picture}(200,110)
\put(20,20){$\bullet$}
\put(5,20){$0$}
\put(20,50){$\bullet$}
\put(5,50){$\lambda_3$}
\put(20,80){$\bullet$}
\put(5,80){$\lambda_2$}
\put(20,110){$\bullet$}
\put(5,110){$\lambda_1$}
\put(40,35){$\circ$}\put(40,65){$\circ$}\put(40,95){$\circ$}
\put(60,50){$\circ$}\put(60,80){$\circ$}
\put(80,65){$\circ$}
\put(25,110){$\line(4,-3){15}$}\put(25,80){$\line(4,-3){15}$}\put(25,50){$\line(4,-3){15}$}
\put(45,95){$\line(4,-3){15}$}\put(45,65){$\line(4,-3){15}$}
\put(65,80){$\line(4,-3){15}$}
\put(60,80 ){$\line(-4,-3){15}$}\put(60,50 ){$\line(-4,-3){15}$}
\put(40,95 ){$\line(-4,-3){15}$}\put(40,65 ){$\line(-4,-3){15}$}\put(40,35 ){$\line(-4,-3){15}$}
\put(80,65){$\line(-4,-3){15}$}
\end{picture}
\caption{The regular $\lie{sl}_{4}$-poset}\label{fig-sl4}
\end{figure}

From here, one can use a similar argument for the $n=4$ case and also for general $n$, that if at least three $m_i$ are non-zero, then there is a star subposet.\\
It remains to show that if one of the three cases in the statement of the proposition is satisfied, then there is no star subposet and hence by Theorem~\ref{thm-equiv}, there is no unimodular equivalence. Suppose $m_3 = \ldots = m_n = 0$, then the poset reduces to Figure~\ref{fig-m12}.
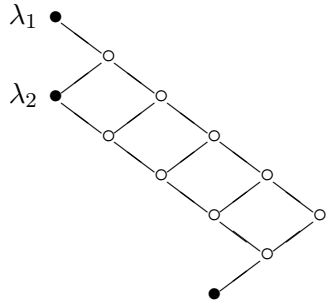
\begin{figure}[ht]
\begin{picture}(200,110)
\put(20,80){$\bullet$}
\put(5,80){$\lambda_2$}
\put(20,110){$\bullet$}
\put(5,110){$\lambda_1$}
\put(40,65){$\circ$}\put(40,95){$\circ$}
\put(60,50){$\circ$}\put(60,80){$\circ$}
\put(80,5){$\bullet$}
\put(80,35){$\circ$}\put(80,65){$\circ$}
\put(100,20){$\circ$}\put(100,50){$\circ$}
\put(120,35){$\circ$}
\put(25,110){$\line(4,-3){15}$}\put(25,80){$\line(4,-3){15}$}
\put(45,95){$\line(4,-3){15}$}\put(45,65){$\line(4,-3){15}$}
\put(65,80){$\line(4,-3){15}$}\put(65,50){$\line(4,-3){15}$}
\put(85,35){$\line(4,-3){15}$}
\put(85,65){$\line(4,-3){15}$}
\put(105,50){$\line(4,-3){15}$}
\put(60,80 ){$\line(-4,-3){15}$}
\put(40,95 ){$\line(-4,-3){15}$}
\put(80,65){$\line(-4,-3){15}$}
\put(100,50 ){$\line(-4,-3){15}$}\put(100,20 ){$\line(-4,-3){15}$}
\put(120,35){$\line(-4,-3){15}$}
\end{picture}
\caption{The $m_1, m_2 \neq 0$-poset}\label{fig-m12}
\end{figure}
This has no star subposet and hence, the corresponding marked chain and marked order polytopes are isomorphic. The same argument is valid for $m_1 = \ldots = m_{n-2} = 0$. For $m_2 = \ldots = m_{n-1} = 0$, the corresponding poset is linear and hence has no star subposet. This finishes the proof of the proposition.
\end{proof}

More general, any reduced decomposition of the longest element of $S_{n+1}$ induces a flat degeneration of $SL_{n+1}/P_\lambda$ (the marked oder /Gelfand-Tsetlin polytope is the particular choice $w_0 =s_1s_2s_1s_3s_2s_1 \ldots $) to toric varieties. In general, these toric varieties are are non-isomorphic in general (\cite{Lit98, AB04}). 
So the natural question is if there exists a reduced decomposition such that the induced toric variety is isomorphic to the toric variety obtained through the marked chain polytope. This will be part of future research.

\subsection{The \texorpdfstring{$\lie{sp}_n$}{symplectic Lie algebra}-case}
Let $(\mathcal{P}, A, \lambda)$ be the marked poset with Hasse diagram Figure~\ref{fig-spn}.\\
\begin{figure}[ht]
\begin{picture}(250,180)
\put(5,0){$0$}\put(5,40){$\lambda_n$}\put(5,130){$\lambda_2$}\put(5,160){$\lambda_1$}\put(5,100){$\lambda_3$}
\put(20,10){$\bullet$}\put(20,40){$\bullet$}\put(20,70){$\vdots$}\put(20,100){$\bullet$}\put(20,130){$\bullet$}\put(20,160){$\bullet$}
\put(60,0){$0$}\put(140,0){$0$}\put(180,0){$0$}\put(220,0){$0$}
\put(40,25){$\circ$}\put(40,55){$\circ$}\put(40,85){$\circ$}\put(40,115){$\circ$}\put(40,145){$\circ$}
\put(60,10){$\bullet$}\put(60,40){$\circ$}\put(60,70){$\circ$}\put(60,100){$\circ$}\put(60,130){$\circ$}
\put(80,25){$\circ$}\put(80,55){$\circ$}\put(80,85){$\circ$}\put(80,115){$\circ$}
\put(100,10){$\cdots$}\put(100,40){$\circ$}\put(100,70){$\circ$}\put(100,100){$\circ$}
\put(120,25){$\circ$}\put(120,55){$\circ$}\put(120,85){$\circ$}
\put(140,10){$\bullet$}\put(140,40){$\circ$}\put(140,70){$\circ$}
\put(160,25){$\circ$}\put(160,55){$\circ$}\put(180,10){$\bullet$}\put(180,40){$\circ$}
\put(200,25){$\circ$}\put(220,10){$\bullet$}

\put(25,160){$\line(4,-3){15}$}\put(25,130){$\line(4,-3){15}$}\put(25,100){$\line(4,-3){15}$}\put(25,40){$\line(4,-3){15}$}
\put(45,145){$\line(4,-3){15}$}\put(45,115){$\line(4,-3){15}$}\put(45,85){$\line(4,-3){15}$}\put(45,55){$\line(4,-3){15}$}\put(45,25){$\line(4,-3){15}$}
\put(65,130){$\line(4,-3){15}$}\put(65,100){$\line(4,-3){15}$}\put(65,70){$\line(4,-3){15}$}\put(65,40){$\line(4,-3){15}$}
\put(85,115){$\line(4,-3){15}$}\put(85,85){$\line(4,-3){15}$}\put(85,55){$\line(4,-3){15}$}
\put(105,100){$\line(4,-3){15}$}\put(105,70){$\line(4,-3){15}$}\put(105,40){$\line(4,-3){15}$}
\put(125,85){$\line(4,-3){15}$}\put(125,55){$\line(4,-3){15}$}\put(125,25){$\line(4,-3){15}$}
\put(145,70){$\line(4,-3){15}$}\put(145,40){$\line(4,-3){15}$}\put(165,55){$\line(4,-3){15}$}\put(165,25){$\line(4,-3){15}$}
\put(185,40){$\line(4,-3){15}$}\put(205,25){$\line(4,-3){15}$}
\put(40,145 ){$\line(-4,-3){15}$}\put(40,115 ){$\line(-4,-3){15}$}\put(40,55 ){$\line(-4,-3){15}$}\put(40,25 ){$\line(-4,-3){15}$}

\put(60,130 ){$\line(-4,-3){15}$}\put(60,100 ){$\line(-4,-3){15}$}\put(60,70){$\line(-4,-3){15}$}\put(60,40){$\line(-4,-3){15}$}
\put(80,115){$\line(-4,-3){15}$}\put(80,85 ){$\line(-4,-3){15}$}\put(80,55 ){$\line(-4,-3){15}$}\put(80,25 ){$\line(-4,-3){15}$}
\put(100,100 ){$\line(-4,-3){15}$}\put(100,70 ){$\line(-4,-3){15}$}\put(100,40 ){$\line(-4,-3){15}$}
\put(120,85){$\line(-4,-3){15}$}\put(120,55 ){$\line(-4,-3){15}$}
\put(140,70 ){$\line(-4,-3){15}$}\put(140,40 ){$\line(-4,-3){15}$}
\put(160,55 ){$\line(-4,-3){15}$}\put(160,25 ){$\line(-4,-3){15}$}
\put(180,40 ){$\line(-4,-3){15}$}
\put(200,25 ){$\line(-4,-3){15}$}
\end{picture}
\caption{The $\lie{sp}_n$-poset}\label{fig-spn}
\end{figure}
The associated marked order polytope is Berenstein-Zelevinsky-Littelmann polytope associated with the partition $\lambda$ (\cite{BZ01, Lit98}). 
The set of lattice points in this polytope parametrizes a basis of the simple $\lie{sp}_{n}$-module with highest weight $\lambda$ (the basis is obtained by restricting $V(\lambda)$ to a specific subalgebra isomorphic to $\lie{sp}_{n-1}$, and iterating this to $\lie{sl}_2$).\\

The marked chain polytope parametrizes a basis of the associated PBW graded module $\operatorname{gr}V(\lambda)$ \cite{FFoL11b}. 

We can read of immediately from the poset and Theorem~\ref{main-thm}, that the toric varieties obtained through the Berenstein-Zelevinsky-Littelmann polytope and the marked chain polytope (via the PBW filtration) are isomorphic if and only if $\lambda_3 = \lambda_4 = \ldots = \lambda_n = 0$ (for $n \geq 4$). 

\subsection{Demazure modules for \texorpdfstring{$\lie{sl}_n$}{the special linear Lie algebra}}
For any sequence $\ell_1 \geq \ldots \geq \ell_p \geq 0$ of integers, we may consider the poset embedded into a triangle, 
such that the lengths of the diagonals are (from left to right) $\ell_1, \ell_2, \ldots, \ell_p$. 
In the example Figure~\ref{fig-dem}, the sequence is $8,7,7,5,3,3,2,2$. We put a maximal marked element on top of the poset, labeled with $m$ and 
add marked elements to each minimal element in the poset, labeled with $0$.\\
\begin{figure}[ht]
\begin{picture}(350,190)
\put(20,0){$0$} \put(20,10){$\bullet$}\put(20,40){$\circ$}
\put(40,55){$\circ$}
\put(60,0){$0$}\put(60,10){$\bullet$}\put(60,40){$\circ$}\put(60,70){$\circ$}
\put(80,0){$0$} \put(80,10){$\bullet$}\put(80,25){$\circ$}\put(80,55){$\circ$}\put(80,85){$\circ$}
\put(100,40){$\circ$}\put(100,70){$\circ$}\put(100,100){$\circ$}
\put(120,55){$\circ$}\put(120,85){$\circ$}\put(120,115){$\circ$}
\put(140,0){$0$} \put(140,10){$\bullet$}\put(140,40){$\circ$}\put(140,70){$\circ$}\put(140,100){$\circ$}\put(140,130){$\circ$}
\put(160,55){$\circ$}\put(160,85){$\circ$}\put(160,115){$\circ$}\put(160,145){$\circ$}\put(160,175){$\bullet$} \put(168,175){$m$} 
\put(180,130){$\circ$}\put(180,100){$\circ$}\put(180,70){$\circ$}
\put(200,115){$\circ$}\put(200,85){$\circ$}\put(200,55){$\circ$}
\put(220,0){$0$} \put(220,100){$\circ$}\put(220,70){$\circ$}\put(220,40){$\circ$}\put(220,10){$\bullet$}
\put(240,85){$\circ$}\put(240,55){$\circ$}
\put(260,70){$\circ$}\put(260,40){$\circ$}
\put(280,0){$0$} \put(280,55){$\circ$}\put(280,25){$\circ$}\put(280,10){$\bullet$}
\put(300,0){$0$} \put(300,40){$\circ$}\put(300,10){$\bullet$}

\put(23,40){$\line(0,-1){25}$}
\put(63,40){$\line(0,-1){25}$}
\put(83,25){$\line(0,-1){10}$}
\put(143,40){$\line(0,-1){25}$}
\put(163,175){$\line(0,-1){25}$}
\put(223,40){$\line(0,-1){25}$}
\put(283,25){$\line(0,-1){10}$}
\put(303,40){$\line(0,-1){25}$}

\put(45,55){$\line(4,-3){15}$}
\put(65,70){$\line(4,-3){15}$}\put(65,40){$\line(4,-3){15}$}
\put(85,85){$\line(4,-3){15}$}\put(85,55){$\line(4,-3){15}$}
\put(105,100){$\line(4,-3){15}$}\put(105,70){$\line(4,-3){15}$}
\put(125,115){$\line(4,-3){15}$}\put(125,85){$\line(4,-3){15}$}\put(125,55){$\line(4,-3){15}$}
\put(145,70){$\line(4,-3){15}$}\put(145,100){$\line(4,-3){15}$}\put(145,130){$\line(4,-3){15}$}
\put(165,85){$\line(4,-3){15}$}\put(165,115){$\line(4,-3){15}$}\put(165,145){$\line(4,-3){15}$}
\put(185,70){$\line(4,-3){15}$}\put(185,100){$\line(4,-3){15}$}\put(185,130){$\line(4,-3){15}$}
\put(205,55){$\line(4,-3){15}$}\put(205,85){$\line(4,-3){15}$}\put(205,115){$\line(4,-3){15}$}
\put(225,70){$\line(4,-3){15}$}\put(225,100){$\line(4,-3){15}$}
\put(245,55){$\line(4,-3){15}$}\put(245,85){$\line(4,-3){15}$}
\put(265,40){$\line(4,-3){15}$}\put(265,70){$\line(4,-3){15}$}
\put(285,55){$\line(4,-3){15}$}

\put(40,55){$\line(-4,-3){15}$}
\put(60,70){$\line(-4,-3){15}$}
\put(80,85 ){$\line(-4,-3){15}$}\put(80,55 ){$\line(-4,-3){15}$}
\put(100,40 ){$\line(-4,-3){15}$}\put(100,70 ){$\line(-4,-3){15}$}\put(100,100 ){$\line(-4,-3){15}$}
\put(120,55 ){$\line(-4,-3){15}$}\put(120,85 ){$\line(-4,-3){15}$}\put(120,115 ){$\line(-4,-3){15}$}
\put(140,130){$\line(-4,-3){15}$}\put(140,100 ){$\line(-4,-3){15}$}\put(140,70 ){$\line(-4,-3){15}$}
\put(160,145 ){$\line(-4,-3){15}$}\put(160,115 ){$\line(-4,-3){15}$}\put(160,85 ){$\line(-4,-3){15}$}\put(160,55 ){$\line(-4,-3){15}$}
\put(180,130){$\line(-4,-3){15}$}\put(180,100 ){$\line(-4,-3){15}$}\put(180,70 ){$\line(-4,-3){15}$}
\put(200,115 ){$\line(-4,-3){15}$}\put(200,85 ){$\line(-4,-3){15}$}
\put(220,100 ){$\line(-4,-3){15}$}\put(220,70 ){$\line(-4,-3){15}$}
\put(240,85 ){$\line(-4,-3){15}$}\put(240,55 ){$\line(-4,-3){15}$}
\put(260,70 ){$\line(-4,-3){15}$}
\put(280,55 ){$\line(-4,-3){15}$}
\put(300,40 ){$\line(-4,-3){15}$}
\end{picture}
\caption{The minuscule Demazure poset}\label{fig-dem}
\end{figure}
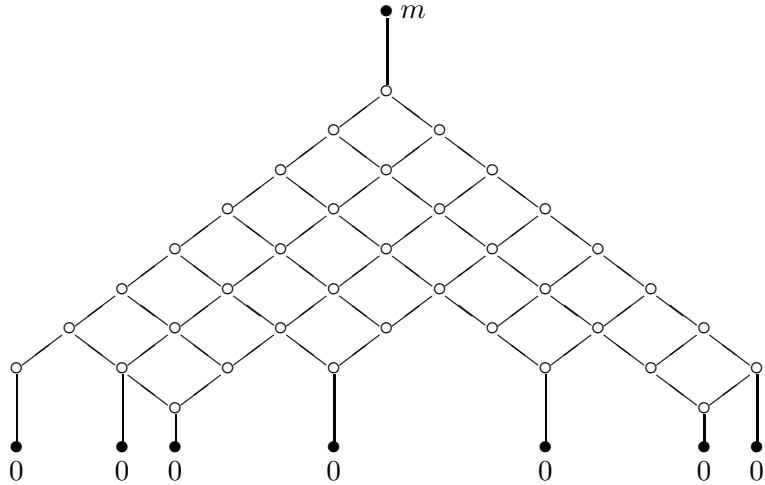
It has been shown in \cite{BF14}, that the lattice points in the marked chain polytope parametrize a basis of the associated 
PBW graded-module of a $\lie{sl}_n$-Demazure module $V_w(\lambda) \subset V(\lambda)$, where $\lambda = m \omega_i$ is rectangular and $w \in S_{n+1}$ is uniquely determined by the sequence $\ell_1 \geq \ldots \geq \ell_p \geq 0$.  
For any Demazure module $V_w(m \omega_i)$ of rectangular highest weight, there exists such a marked poset. The marked order polytope is a maximal Kogan 
face of type $w$ in the Gelfand-Tsetlin polytope of highest weight $m \omega_i$.\\

In \cite{BF14}, the number of facets in both polytopes has been computed as well as the theorem on unimodular equivalence has been proved. 
This has been done using the results by Stanley (\cite{Sta86}) and Hibi-Li (\cite{HL12}), since in this case the marked poset polytopes are in fact dilations of $0-1$ polytopes\\

In \cite{Fou14b}, a marked poset was given for Demazure modules of arbitrary highest weight in the $\lie{sl}_n$-case for triangular Weyl group elements. 
Here, the condition triangular ensures that the naturally obtained relations (from representation theory) can be translated into chains within the poset. 
An interpretation of the marked order polytope is missing but a close connection to Kogan faces (\cite{Kog00, KST12}) is expected.
\subsection{Further cominuscule cases}
Let $\lie g$ be of any finite type (\cite{Car05}) and $\omega_i$ a cominuscule fundamental weight. Then in \cite{BD14} a (marked) poset is provided such that the lattice points in the chain polytope parametrize a basis of the associated PBW graded module $\operatorname{gr} V(m\omega_i)$. 
Again, an interpretation of the order polytope is missing. One may expect a relation to string polytopes defined in \cite{Lit98} for exceptional cases.

\subsection{Indecomposables correspond to fundamental weights}
There is a common idea in all the paper about PBW graded modules on how to show that the lattice points in the marked chain polytope parametrize a basis. 
First, one has to show that there is a total homogeneous order on the monomials such that the lattice points give a spanning set of the associated graded module $\operatorname{gr}^t V(\lambda)$ (the graded components are $0$ or $1$-dimensional).\\
In order to prove that they give linear independent monomials, one uses that 
\[V(\lambda + \mu) = U(\lie u).v_\lambda \otimes v_\mu \subset V(\lambda) \otimes V(\mu)\] 
and that if $M_\lambda$ is a set of monomials which is linear independent in $\operatorname{gr}^t V(\lambda)$ ($M_\mu$ similarly for$\operatorname{gr}^t V(\mu)$), then the set of products $M_\lambda \dot M_\mu$ is linear independent in $V(\lambda + \mu)$ \cite{FFoL13}. \\
So it suffices to show that the lattice points marked chain polytope for $\lambda + \mu$ is the Minkowski sum of the lattice points for $\lambda$ and $\mu$. Then it would suffices to check the linear independence for the ''smaller'' weights $\lambda$ and $\mu$. \\
This has been done in the aforementioned paper case by case, while now Theorem~\ref{main-thm2} gives a unified proof for all marked chain polytopes. Namely, one has to check the linear independence for $\mathcal{C}$-indecomposable marked polytopes only. Then the rest follows from the Minkowski sum property of the marked chain polytopes.\
In the cases considered above, the $\mathcal{C}$-indecomposables correspond to fundamental weights.


\section{Facets of the marked poset polytopes}\label{Section-Facets}
Let $(\mathcal{P}, A, \lambda)$ be a marked poset. We would like to compute the number of facets of the marked order and the marked chain polytope. \\
To simplify the computation we delete all obvious redundant cover relations from the marked poset without changing the number of facets:\\ 
Let $a\neq b \in A$, $a \prec x_1 \prec \ldots  \prec x_s \prec b$ and $\lambda_a = \lambda_b$, then in the marked chain polytope $x_1 = \ldots = x_s = 0$ and in the marked order polytope $x_1 = \ldots = x_s = \lambda_b$. \\
We construct a new marked poset by introducing a new vertex $c \in A$ with cover relations: for all $x \in \mathcal{P}$ such that there exists $y \in \{ a,x_1, \ldots, x_s, b \}$ with $y$ covers $x$ (resp. $x$ covers $y$), we set  $x \prec c$ (resp. $c \prec x$). We mark this new vertex by $\lambda_b$.\\ 
The new marked poset is now $(\mathcal{P}\setminus\{a , x_1, \ldots, x_s, b\} \cup \{c\}, A \setminus\{a,b\} \cup \{c\}, \overline{\lambda})$,
 where $\overline{\lambda}$ is the marking vector obtained from $\lambda$ by adding $\lambda_c$ and deleting $\lambda_b, \lambda_a$. We iterate this until for all $a \prec b \in A:\lambda_a < \lambda_b$. \\
This is obviously a poset, since the new defined partial order is again transitive (we have divided our poset into several ''free'' vertices and several connected clusters whose entries in the marked poset polytopes are fixed and identified these cluster with a new vertex).\\
Suppose now $x$ covers $a$ and further, there exists $b \prec x$ with $\lambda_b \geq  \lambda_a$, where $a,b \in A, x \in \mathcal{P}\setminus A$. 
Then we delete the cover relation $a \prec x$ from our poset. On the other hand, suppose $a$ covers $x$ and there exists $x \prec b$ with $\lambda_b \leq \lambda_a$, where again $a,b \in A, x \in \mathcal{P}\setminus A$. Then we delete the cover relation $x \prec a$ from the poset. To visualize this: the following two configurations are changed such that the relation between $x$ and $a$ is deleted from the poset:
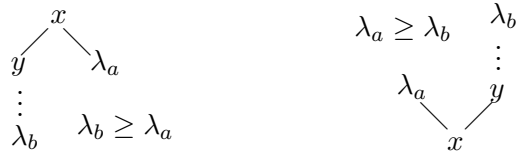
\begin{picture}(300,70)
\put(100,4){$\lambda_b$}
\put(114,48){$\line(-1,-1){10}$}
\put(115,50){$x$}
\put(120,48){$\line(1,-1){10}$}
\put(102,15){$\vdots$}
\put(100,32){$y$}
\put(130,32){$\lambda_a$}
\put(125,8){$\lambda_b \geq \lambda_a$}

\put(246,23){$\lambda_a$}
\put(255,20){$\line(1,-1){10}$}
\put(265,2){$x$}
\put(283,32){$\vdots$}
\put(281,50){$\lambda_b$}
\put(282,20){$\line(-1,-1){10}$}
\put(281,22){$y$}
\put(230,45){$\lambda_a \geq \lambda_b$}
\end{picture}\\
We iterated this process (which is finite) until these cases do not appear in our poset.\\
There is one last step before calling the marked poset regular: We may assume that $\lambda_a \neq \lambda_b$ for all $a \neq b \in A$. 
Suppose $\lambda_a = \lambda_b$, then we simply introduce a new vertex $c \in A$, with $\lambda_c = \lambda_a$ and the relations induced from $a$ and $b$. And we delete all cover relations which contain elements from $A$ only, e.g. cover of the form $a \prec b$ with $a,b \in A$.\\
We call the resulting marked poset the \textit{regular}  marked poset associated to $(\mathcal{P}, A, \lambda)$.
The following is obvious:
\begin{prop}
The number of facets in the marked order polytope as well as in the marked chain polytope associated to the regular marked poset is the same as the number of facets in $\mathcal{O}(\mathcal{P}, A)_\lambda$ (resp. $\mathcal{C}(\mathcal{P}, A)_\lambda$).
\end{prop}

So to compute the number of facets in the marked poset polytopes we can assume that $(\mathcal{P}, A, \lambda)$ is already regular. 
Then denote $h(\mathcal{P})$ the number of cover relations in $\mathcal{P}$, and for $a,b \in A$, denote $c_\mathcal{P}(a,b)$ the number of saturated chains $a \prec x_1 \prec  \ldots \prec x_s \prec b$ (here $x_i \in \mathcal{P}\setminus A$ and saturated means that in each step there is no $y \in \mathcal{P}\setminus A$ such that $x_i \prec y \prec x_{i+1}$). 
\begin{lem}\label{lem-num-facets} For any regular marked poset $(\mathcal{P}, A, \lambda)$ we have
\[
\sharp \{\text{ facets of } \mathcal{O}(\mathcal{P}, A)_\lambda\} = h(\mathcal{P}) \; \; , \; \; \sharp \{\text{ facets of } \mathcal{C}(\mathcal{P}, A)_\lambda \}= |\mathcal{P}\setminus A| + \sum_{a \prec b \in A} c_\mathcal{P}(a,b).
\]
\end{lem}

\begin{thm}
For any marked poset $(\mathcal{P}, A, \lambda)$, the number of facets in $\mathcal{O}(\mathcal{P}, A)_\lambda$ is less or equal to the number of facets in $\mathcal{O}(\mathcal{P}, A)_\lambda$.
\end{thm}
\begin{proof}
We generalize here the idea of the proof in \cite{HL12} for chain and order polytopes (e.g. $\lambda_a \in \{0,1\}$ for all $a \in A$). 
So we will use induction on the number of elements in $\mathcal{P}$. Let $x \in \mathcal{P}\setminus A$ be a minimal element in the subposet $\mathcal{P}\setminus A$ and suppose $x$ is not maximal (in this poset). If $x$ would be minimal and maximal, then in both polytopes, there are exactly two facets induced from the covering relations $a \prec x \prec b$, so we can omit this case here. Since the marked poset is regular, there is a unique $a \in A$ such that $x$ covers $a$, moreover if $b \prec x$ then $b = a$ since $x$ is minimal. \\
We define the marked poset $(\mathcal{P}\setminus\{x\}, A, \lambda)$ to be the full subposet without the vertex $x$. We can assume that we have chosen $x$ such that $\lambda_a$, for the covered $a$, is maximal. 
Then the marked poset  $(\mathcal{P}\setminus\{x\}, A, \lambda)$ is again regular (the first part of the reduction does not apply in this case, while the second part of the reduction is not necessary due to the fact that $\lambda_a$ is maximal). By induction we know:
\[
h(\mathcal{P} \setminus\{x\}) \leq  |\mathcal{P}\setminus A| -1 + \sum_{a \prec b \, \in \,  A} c_{\mathcal{P}\setminus\{x\}}(a,b).
\]
We will compute the difference $h(\mathcal{P}) - h(\mathcal{P} \setminus\{x\})$.\\
We denote $D = \{y_1, \ldots, y_s\}$ the set of elements in $\mathcal{P}\setminus A$ such that $y_i$ covers $x$ and there exists a saturated chain 
$b \prec z_1 \prec \ldots \prec z_t \prec y_i$ with $z_{i_j} \in \mathcal{P}\setminus A$, $\lambda_b \geq \lambda_a$ and $z_t \neq x$ (not that $t$ might be $0$ and then $z_{i_t} = b$, also $a = b$ is possible). For each $y_i \in D$, 
we denote $D_i$ the number of saturated chains in $\mathcal{P}\setminus A$ starting in $y_i$ ($y_i \prec z_1 \prec \ldots  \prec z_t$). Certainly $D_i \geq 1$ as $y$ is in any such chain. Then
\[
h(\mathcal{P} \setminus\{x\}) = h(\mathcal{P}) - 1 - |D|.
\]
Further for $a \prec c$:
\[
c_{\mathcal{P}\setminus\{x\}}(a,c) = c_{\mathcal{P}}(a,c) - \sum_{a \prec y_i  \prec c} D_i
\]
and for all $b,c \in A$ with $b \neq a$:
\[
c_{\mathcal{P}\setminus\{x\}}(b,c) = c_{\mathcal{P}}(b,c).
\]
Since any $y_i \in D$ is part of at least one maximal chain $a \prec x \prec y_i \prec \ldots \prec c$ for some $c$, we have
\begin{eqnarray}\label{d-eq}
|D| \leq \sum_{a \prec c \in A}  \sum_{a \prec y_i \prec c} D_i.
\end{eqnarray}
Concluding we have 
\[
h(\mathcal{P}) =  h(\mathcal{P} \setminus\{x\}) + 1 + |D| \leq \left(|\mathcal{P}\setminus A| -1 + \sum_{b \prec c \in A} c_{\mathcal{P}\setminus\{x\}}(b,c)\right) + 1 + |D| \leq  |\mathcal{P}\setminus A|  + \sum_{b \prec c \in A} c_{\mathcal{P}}(b,c) . 
\]
\end{proof}

We can deduce easily that the number of facets in both polytopes is equal if and only if the inequality in \eqref{d-eq} is in fact an equality in each step. 
\begin{coro}\label{cor-facet}
Let $(\mathcal{P}, A, \lambda)$ be marked polytope. Then the number of facets in $\mathcal{O}(\mathcal{P}, A)_\lambda$ is equal to the number of facets in $\mathcal{C}(\mathcal{P}, A)_\lambda$ if and only if the regular marked poset has no star relation.
\end{coro}
\begin{proof}
Suppose we do not have a star relation. Let $x$ be chosen as in the proof of the theorem and suppose $y$ covers $x$. If $y \notin D$, then deleting $x$ has no influence on the number of facets involving $y$. So we can assume $y \in D$.\\
Then, since the subposet does not appear in the regular poset, there is a unique saturated chain starting in $y \prec ...$:\\
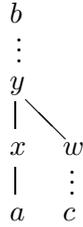
\begin{picture}(200,95)
\put(125,05){$a$}
\put(127,25){$\line(0,-1){10}$}
\put(125,30){$x$}
\put(127,50){$\line(0,-1){10}$}
\put(125,55){$y$}
\put(131,51){$\line(1,-1){15}$}
\put(145,30){$w$}
\put(145,05){$c$}
\put(147,15){$\vdots$}
\put(127,65){$\vdots$}
\put(125,80){$b$}

\end{picture}\\
This implies first, that $D_y =1$ and second there is a unique $c \in A$ with $a \prec y \prec c$ ($a$ again as in the proof of the theorem). This implies that 
\[
\sum_{a \prec c \in A} \sum_{a \prec y \prec c} D_i = \sum D_i = |D|.
\]
This implies equality in this induction step and hence the corollary.\\
Suppose there exists a star relation with $x_1, x_2 \prec x \prec x_3, x_4$ in the reduced poset. First of all, we can assume that $x$ covers $x_1$ and $x_2$. If neither $x_1$ nor $x_2$ is minimal in $\mathcal{P}\setminus A$, we can use the same induction as in the theorem (induction on the cardinality of $\mathcal{P}$) to prove the corollary. So we can wlog that $x_1$ is minimal in $\mathcal{P}\setminus A$. Further, we can assume that at least one of $x_3, x_4 \notin A$, so say $x_3 \in \mathcal{P}\setminus A$.\\
We perform the induction step with the element $x_1$. Then $x$ covers $x_1$, $x \in D$ (since $... x_2 \prec x$ is a chain) and we have $D_x \geq 2$, since $x_3$ and $x_4$ are incomparable\\
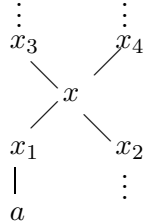
\begin{picture}(200,95)
\put(125,05){$a$}
\put(127,25){$\line(0,-1){10}$}
\put(125,30){$x_1$}
\put(143,50){$\line(-1,-1){10}$}
\put(143,55){$\line(-1,1){10}$}
\put(125,70){$x_3$}
\put(145,50){$x$}
\put(163,35){$\line(-1,1){10}$}
\put(165,30){$x_2$}
\put(165,70){$x_4$}
\put(167,70){$\line(-1,-1){10}$}
\put(167,13){$\vdots$}
\put(127,78){$\vdots$}
\put(167,78){$\vdots$}
\end{picture}\\
and hence there are at most two saturated chains $x \prec x_3 \prec \ldots, x \prec x_4 \prec \ldots$. But this implies that \eqref{d-eq} is a proper inequality and then the corollary follows.
\end{proof}


\section{Unimodular equivalence}\label{Section-Uni}
Let $P, Q \subset \mathbb{R}^N$ two convex polytopes. A \textit{unimodular equivalence} of $P, Q$ is a map $\Psi:  \mathbb{R}^N \longrightarrow  \mathbb{R}^N$, such that $\Psi(v) = M.v + w$, where $M \in M_{n \times n}(\bz)$ with $\operatorname{det }M  = \pm 1$ and $w$ is a constant vector and $\Psi(P) = Q$. Certainly, any two unimodular equivalent polytopes have the same f-vector. In this section we will prove
\begin{thm}
Let $(\mathcal{P}, A, \lambda)$ be a regular marked poset. Then  $\mathcal{O}(\mathcal{P}, A)_\lambda$ is unimodular equivalent to $\mathcal{C}(\mathcal{P}, A)_\lambda$ if and only if $\mathcal{P}$ has no star relation.
\end{thm}
\begin{proof}
Recall that regular implies that $\lambda_a \neq \lambda_b$ for $a \neq b$ in $A$. Now suppose $\mathcal{P}$ has such a star relation, then by Corollary~\ref{cor-facet} the number of facets in the marked order and the marked chain polytope is different, so they are not unimodular equivalent. \\
If the marked poset has no star relation, then we construct a map $\Psi: \mathcal{C}(\mathcal{P}, A)_\lambda \longrightarrow \mathcal{O}(\mathcal{P}, A)_\lambda$ defining the unimodular equivalence. \\
By assumption, for any $x \in \mathcal{P}\setminus A$ such that there exists two distinct saturated chains $x \prec \ldots \prec b, x \prec \ldots \prec c$ for some $b,c \in A$, there exists at most one saturated chain $a \prec \ldots \prec x$. Any maximal $x$ is covered by exactly one element $b \in A$ and any minimal $x$ covers exactly one $a \in A$. We define $\Psi $ as follows for $x \in \mathcal{P}\setminus A$:
\begin{enumerate}
\item If $x$ is maximal but not minimal in $\mathcal{P}\setminus A$:  $\Psi(\bs_x) = \lambda_b - \bs_x $.
\item If $x$ is minimal in $\mathcal{P}\setminus A$: $\Psi(\bs_x) = \bs_x + \lambda_a$.
\item Else if $x \prec y_1 \prec \ldots \prec y_s \prec b$ is the unique saturated chain starting in $x$: $\Psi(x) = \lambda_b - x - y_1 - \ldots - y_s$.
\item Else if $a \prec y_1 \prec \ldots  \prec y_s \prec x$ is the unique saturated chain ending in $x$: $\Psi(x) = x + y_1 + \ldots + y_s + \lambda_a$.
\end{enumerate}
So we can write $\Psi = U + w$, where $ U \in  M_{\mathcal{P}\setminus A}(\bz)$ with $\operatorname{det } U = \pm 1$ and $w \in \mathbb{R}^{\mathcal{P} \setminus A}$. It remains to show that $\Psi$ restricted to $ \mathcal{C}(\mathcal{P}, A)_\lambda$ induces a bijection to  $\mathcal{O}(\mathcal{P}, A)_\lambda$.\\
Let 
\[
a = x_0 \prec x_1 \prec \ldots \prec x_{k-1} \prec x_k \prec \ldots \prec x_s \prec x_{s+1} = b
\]
be a saturated chain, such that $x_s$ is maximal, $x_1$ is minimal in $\mathcal{P}\setminus A$ and $x_k$ is the minimal element in the chain having a unique saturated chain starting in $x_k$. Then $\Psi$ maps this chain to
\[
(\bs_{x_1} + \lambda_a, \ldots, \bs_{x_1}+ \ldots + \bs_{x_{k-1}} + \lambda_a, \lambda_b - \bs_{x_s} - \ldots  - \bs_{x_k}, \ldots, \lambda_b - \bs_{x_s}).
\]
We see that for $i \geq k$ (resp. $i < k$):
\[
\lambda_b - \bs_{x_s} - \ldots  - \bs_{x_i} \leq \lambda_b - \bs_{x_s} - \ldots  - \bs_{x_{i+1}} \text{ and } \bs_{x_i}+ \ldots + \bs_{x_{1}} + \lambda_a \geq \bs_{x_{i-1}}+ \ldots + \bs_{x_{1}} + \lambda_a.
\] 
Further, since $\lambda_b - \lambda_a \geq \bs_{x_1} + \ldots \bs_{x_s}$, we have $\bs_{x_1}+ \ldots + \bs_{x_{k-1}} + \lambda_a \leq \lambda_b - \bs_{x_s} - \ldots  - \bs_{x_k}$. This implies that $\Psi(\mathcal{C}(\mathcal{P}, A)_\lambda) \subset \mathcal{O}(\mathcal{P}, A)_\lambda$. 
The inverse of $\Psi$ is given by 
\[
\Psi^{-1}(\bt_{x_i}) = \begin{cases} \lambda_b - \bt_{x_i} \text{ if } i = s \\ \bt_{x_i} - \lambda_a \text{ if } i=1 \\ \bt_{x_{i+1}} - \bt_{x_i} \text{ if } i \geq k \\ \bt_{x_i} - \bt_{x_{i-1}} \text{ if } i < k \end{cases}
\]
and with a similar argument we see that $\Psi^{-1}(\mathcal{O}(\mathcal{P}, A)_\lambda) \subset \mathcal{C}(\mathcal{P}, A)_\lambda$. Now it is straightforward to see that $\Psi^{-1}(\bt_{x_i} = \bt_{x_{i+1}})$ is the facet of the marked order polytope given by
\begin{itemize}
\item $\bs_{x_i} = 0$ if $s \geq i \geq k$.
\item $\bs_{x_{i-1}} = 0$ if $0 \leq i < k-1$.
\item $\lambda_b - \lambda_a = \bs_{x_1} + \ldots \bs_{x_s}$ if $i = k-1$.
\end{itemize}
Hence, faces are mapped to facets and $\Psi$ defines a unimodular equivalence.
\end{proof}


\section{Indecomposables}\label{Section-Inde}
In this section we want to study the indecomposable marked chain polytopes. We will prove first:
\begin{lem}
Let $\lambda \in (\mathcal{Q}_A)_0$ and suppose there exists $a \in A$ with $\lambda_a > 1$. Then $\lambda$ is not $\mathcal{C}$-indecomposable.
\end{lem}
\begin{proof}
Let $\lambda \in (\mathcal{Q}_A)_0$ and let $\omega \in (\mathcal{Q}_A)_0$ defined as in:
\[
\omega_b := \begin{cases} 1 \text{ if } \lambda_b \neq 0 \\ 0 \text{ if } \lambda_b = 0 \end{cases}.
\]
We will prove that the lattice points in the marked chain polytopes can be decomposed into lattice points for $\omega$ and lattice points for $\lambda - \omega$. This implies the lemma.\\

\noindent\textit{Proof:}
Let $\bs \in S_{\mathcal{C}}(\lambda)$ and define
\[
M = \{ p \in \mathcal{P}\setminus A \mid s_b \neq 0 \} \cup \{ b \in A \mid  \lambda_b \neq 0 \}.
\]
Then $M$ is by restriction again a partially ordered set and we denote $M_{min}$ the subset of minimal elements in $M$. We define $\bt \in \bz^{\mathcal{P} \setminus A}$ as
\begin{eqnarray}
\bt_b := \begin{cases} 1 \text{ if } b \in M_{min} \\ 0 \text{ if } b \notin M_{min} \end{cases}
\end{eqnarray}
Claim: $\bt \in S_{\mathcal{C}}(\omega)$.\\
\textit{Proof of the claim:} Let $\bp = a \prec b_1 \prec \ldots \prec b_r \prec b$ be a chain with $a,b \in A$, $b_i \in \mathcal{P}\setminus A$. Since $M_{min}$ is the set of minimal elements, we have $\sum \bs_{b_i} \leq 1$ and even more, if $\lambda_a \neq 0$, then $\sum \bs_{b_i} = 0$. So the only case, that needs to be considered is $\lambda_b = \lambda_a = 0$. But in this case $\bs_x = 0$ for all $a \prec x \prec b$. This proves the claim.\\

Claim: $\bs - \bt \in S_{\mathcal{C}}(\lambda - \omega)$.\\
\textit{Proof of the claim:}  Let $\bp = a \prec  b_1 \prec \ldots \prec b_r \prec b$ be a chain with $a,b \in A$, $b_i \in \mathcal{P}\setminus A$. We have to prove that
\[
\sum \bs_{b_i} - \sum \bt_{b_i} \leq \lambda_b - \lambda_a - \omega_b + \omega_a.
\]
The left hand side is certainly less or equals to $\sum \bs_{b_i}$ and by assumption $\sum \bs_{b_i} \leq \lambda_b - \lambda_a$. So if $\omega_b = \omega_a$, then there is nothing to show. Suppose now $\omega_b \neq \omega_a$, then $\omega_b =1, \omega_a =0, \lambda_b \neq 0, \lambda_a = 0$. \\
Suppose there exists $b_i$ with $\bs_{b_i} \neq 0$ (if all $\bs_{b_i}$ are equal to $0$, then the claim follows immediately). We may assume that $b_i$ is the minimal element in the chain $\bp$ with $\bs_{b_i} \neq 0$. There are two cases to be considered:\\
First, $b_i \in M_{min}$, then $\sum \bt_{b_i} > 1$ and so $\sum \bs_{b_i} - \sum \bt_{b_i} \leq \lambda_b - \lambda_a - \omega_b + \omega_a.$\\
Second $b_i \notin M_{min}$, then there exists $c  \prec b_i$ with either $\bs_c \neq 0$ or $\lambda_c \neq 0$, we may assume that $c$ is minimal. In the first case, $\bs_c \neq 0$, we consider the chain $\bp' = c \prec b_i \prec b_{i+1} \prec \ldots \prec b_r \prec b$. We can extend this chain by adding an element $d$ from $A$ which is less than the elements from $\bp'$ (recall that $A$ contains all minimal elements). Then by assumption
\[ 
\bs_c + \bs_{b_i} + \ldots + \bs_{b_r} \leq \lambda_b - \lambda_d \text{ and } \bs_c =1
\]
Since $\lambda_a = 0$ we have $\lambda_b - \lambda_d \leq \lambda_b - \lambda_a$ and so
\[
\bs_{b_{1}} + \ldots + \bs_{b_r} = \bs_{b_i} + \ldots + \bs_{b_r} < \lambda_b - \lambda_a
\]
which implies the claim in this case.\\
The remaining case is $\lambda_c \neq 0$, then we consider the chain $c \prec b_i \prec \ldots \prec b_r \prec b$. By assumption we have
\[
\bs_{b_i} + \ldots + \bs_{b_r} \leq \lambda_b - \lambda_c < \lambda_b - \lambda_a
\]
(since $\lambda_a = 0$) which again implies the claim.\\
Combining both claims we have
\[
S_{\mathcal{C}}(\lambda) = S_{\mathcal{C}}(\lambda - \omega) + S_{\mathcal{C}}(\omega)
\]
and the marked chain statement of the main theorem follows now by downward induction on $\lambda$.\\

\end{proof}

\subsection{}
We continue the proof of Theorem~\ref{main-thm2}. 
\begin{prop}
Let $\lambda \in (\mathcal{Q}_A)_0$ with $\lambda_a \in \{0,1\}$. If the reduced poset of $\mathcal{P}$ is not connected then $\lambda$ is not $\mathcal{C}$-indecomposable. 
\end{prop}
\begin{proof}
So suppose the reduced poset decomposes into a disjoint union of the Hasse diagram, e.g.
\[
\overline{(\mathcal{P}, A, \lambda)} = P_1 \cup P_2
\]
and set $A_j = A \cap P_j$. We define further
\[
\lambda^j_a = \begin{cases} \lambda_a \text{ if } a \in A_j \\ 0 \text{ else } \end{cases}
\]
First of all, $\lambda^j \in (\mathcal{Q}_A)_0$: $\lambda^j_a \in \{0,1\}$ since $\lambda_a \in \{0,1\}$. Further, suppose $\exists \, a \prec b$ with $\lambda^j_a =1$. Since $a \prec b$ and $\lambda^j_a =1 = \lambda_a$, we have $\lambda_b = 1$. 
This implies that $a,b \in  \overline{(\mathcal{P}, A, \lambda)} $ and hence $a,b$ are in a common connected component, and so $a,b \in P_j$. This implies $\lambda^j_b =1$. \\
Further $\lambda = \lambda^1 + \lambda^2$, this is easy to see: Suppose $\lambda_b = 1$, then $b \in \overline{(\mathcal{P}, A, \lambda)}$ and so $b \in A^1 \cup A^2$, and so $\lambda^1_b + \lambda^2_b =1$. Certainly, if $\lambda_b = 0$, then $\lambda^1_b + \lambda^2_b =0$.\\

Now, we have to show that the lattice points in marked chain polytopes are the Minkowski sums of the lattice points in the polytopes corresponding to $\lambda^1, \lambda^2$.\\

\noindent\textit{Proof:}
Let $x_1 \prec \ldots \prec x_s$ be a chain in $\mathcal{P}\setminus A$. Let $\bs \in  S_{\mathcal{C}}(\lambda^1), \bt \in  S_{\mathcal{C}}(\lambda^2)$, then either $\sum  \bs_{x_i} = 0$ or $\sum \bt_{x_i} = 0$. To see this, suppose $\sum \bs_{x_i} \neq 0$, then there exists $a,b \in A^1$ with $\lambda_a^1 =0, \lambda_b^1 = 1$ and 
\[
a \prec x_1 \prec \ldots \prec x_s \prec b
\]
and even more $\lambda^2_b = 0$. But this implies $\sum \bt_{x_i} \leq \lambda^2_b - \lambda_a^2 = 0$. So we have $\bs + \bt \in  S_{\mathcal{O}}(\lambda)$.\\
On the other hand, let $\bs \in  S_{\mathcal{C}}(\lambda)$ and define
\[
\bs^j_x := \begin{cases} s_x \text{ if } x \in P^j \setminus A^j \\ 0 \text{ else } \end{cases}
\]
then $\bs = \bs^1 + \bs^2$. And we are left to show that $\bs^j \in S_{\mathcal{C}}(\lambda^j)$. So suppose $\bs^j_{x} \neq 0$, then for all chains $x_1 \prec \ldots \prec x \prec \ldots \prec x_s$ in $\mathcal{P}\setminus A$, we have $\bs^j_{x_k} = 0$, since $\bs \in S_{\mathcal{C}}(\lambda)$. Further, there exists $b \in A$ with $\lambda_b = 1$ and $x \prec b$, and if $a \prec x$, $a \in A$, then $\lambda_a = \lambda^j_a = 0$. This implies that $x,b \in  \overline{(\mathcal{P}, A, \lambda)}$ and hence in a common connected component, so $\lambda^j_b =1$, which finishes the proof.\\
\end{proof}

\subsection{}
The very last thing in the proof of Theorem~\ref{main-thm2}:
\begin{prop}
Let $\lambda \in (\mathcal{Q}_A)_0$ with $\lambda_a \in \{0,1\}$. Suppose $\overline{(\mathcal{P}, A, \lambda)}$ is connected, then $\lambda$ is $\mathcal{C}$-indecomposable.\\
\end{prop}
\begin{proof}
Suppose $\lambda$ is not $\mathcal{C}$-indecomposable, then $\lambda = \lambda^1 + \lambda^2$, both are non-zero and 
\[
S_{\mathcal{C}}(\lambda) = S_{\mathcal{C}}(\lambda^1) + S_{\mathcal{C}}(\lambda^2).
\]
We can naturally view $\overline{(\mathcal{P}, A, \lambda^j)} \subset \overline{(\mathcal{P}, A, \lambda)}$ as a subposet, especially as a subset. Let $a \in A$ with $a \in \overline{(\mathcal{P}, A, \lambda)}$, this is equivalent to $\lambda_a = 1$ and hence either $\lambda^1_a =1$ or $\lambda^2_a =1$, which implies $a \in \overline{(\mathcal{P}, A, \lambda^1)} \cup \overline{(\mathcal{P}, A, \lambda^2)}$. \\
Further, if $x \in \overline{(\mathcal{P}, A, \lambda^1)} \cup \overline{(\mathcal{P}, A, \lambda^2)}$, then $y \in \overline{(\mathcal{P}, A, \lambda^1)} \cup \overline{(\mathcal{P}, A, \lambda^2)}$ for all $x \prec y$. \\
Let $x$ be minimal such that $x \in \overline{(\mathcal{P}, A, \lambda)}$ but $x \notin \overline{(\mathcal{P}, A, \lambda^1)} \cup \overline{(\mathcal{P}, A, \lambda^2)}$. Then there is no $a \in A$ with $a \prec x$ and $\lambda_a =1$, and there is no $b \in A$ with $x \prec b$ and $\lambda_b = 0$. We set
\[
\bt^x_y = \begin{cases} 1 \text{ if } y = x \\ 0 \text{ else } \end{cases}
\]
Then $\bt^x \in S_{\mathcal{C}}(\lambda)$, since if $a \prec \ldots \prec x \prec \ldots \prec b$ is a chain, then $\lambda_b - \lambda_a = 1$. But since $x \notin \overline{(\mathcal{P}, A, \lambda^1)} \cup \overline{(\mathcal{P}, A, \lambda^2)}$, we have $\bs_x = 0$ for all $\bs \in S_{\mathcal{C}}(\lambda^1), S_{\mathcal{C}}(\lambda^2)$, this implies that $\bs_x = 0$ for all $\bs \in S_{\mathcal{C}}(\lambda^1) + S_{\mathcal{C}}(\lambda^2)$, which is a contradiction since by assumption $\bt^x \in S_{\mathcal{C}}(\lambda^1) + S_{\mathcal{C}}(\lambda^2)$. This implies that 
\[
\overline{(\mathcal{P}, A, \lambda)} =  \overline{(\mathcal{P}, A, \lambda^1)} \cup \overline{(\mathcal{P}, A, \lambda^2)}
\]
Thus we have shown that if $\lambda$ is not $\mathcal{C}$-indecomposable, then $\overline{(\mathcal{P}, A, \lambda)}$ is not connected. 
\end{proof}

\end{document}